\documentclass[12pt]{article}
\usepackage{graphicx} 

\usepackage{fullpage,comment}
\usepackage{amssymb}%
\usepackage{amsmath}%
\usepackage{amsfonts}%
\usepackage{amsthm}
\usepackage{changepage}
\usepackage{enumitem}
\usepackage{dsfont}
\usepackage{thm-restate}
\usepackage{xcolor}
\usepackage[normalem]{ulem}
\newtheorem{theorem}{Theorem}[section]
\newtheorem{proposition}[theorem]{Proposition}

\newtheorem{corollary}[theorem]{Corollary}

\newtheorem{question}[theorem]{Question}

\newtheorem{observation}[theorem]{Observation}

\newtheorem*{lemma*}{Lemma}
\newtheorem{claim}[theorem]{Claim}
\newtheorem*{claim*}{Claim}
\newtheorem{problem}[theorem]{Problem}
\newtheorem{construction}[theorem]{Construction}

\theoremstyle{remark}
\newtheorem*{remark}{Remark}

\newcommand{\bb}[1]{\mathbb{#1}}
\newcommand{\ca}[1]{\mathcal{#1}}
\newcommand{\bc}{,\allowbreak}
\newcommand{\ff}[2]{\left\lfloor\frac{#1}{#2}\right\rfloor}

\newcommand{\mr}[1]{\mathrm{#1}}
\newcommand{\cf}[2]{\left\lceil\frac{#1}{#2}\right\rceil}

\DeclareMathOperator{\sat}{sat}
\DeclareMathOperator{\wsat}{wsat}
\DeclareMathOperator{\prsat}{sat^\ast}
\DeclareMathOperator{\ssat}{ssat}

\DeclareMathOperator{\ex}{ex}
\DeclareMathOperator{\rex}{ex^\ast}

\setenumerate[0]{label=(\arabic*)}

\title{Improved bounds for proper rainbow saturation}
\author{Andrew Lane\footnotemark[1]\thanks{Department of Mathematics and Statistics, University of Victoria, Canada. \\ Email: \texttt{\{andrewlane,nmorrison\}@uvic.ca}. } \thanks{Research supported by the Jamie Cassels Undergraduate Research Awards and Science Undergraduate Research Awards, University of Victoria.} 
    \and Natasha Morrison\footnotemark[1] \thanks{Research supported by NSERC Discovery Grant RGPIN-2021-02511 and NSERC Early Career Supplement DGECR-2021-00047.}
	}
\date{\today}

\begin{document}

\maketitle

\begin{abstract}
Given a graph $H$, we say that a graph $G$ is \emph{properly rainbow $H$-saturated} if: (1) There is a proper edge colouring of $G$ containing no rainbow copy of $H$; (2) For every $e \notin E(G)$, every proper edge colouring of $G+e$ contains a rainbow copy of $H$. The \emph{proper rainbow saturation number} $\prsat(n,H)$ is the minimum number of edges in a properly rainbow $H$-saturated graph. In this paper we use connections to the classical saturation and semi-saturation numbers to provide new upper bounds on $\prsat(n,H)$ for general cliques, cycles, and complete bipartite graphs. We also provide some general lower bounds on $\prsat(n,H)$ and explore several other interesting directions. 
\end{abstract}

\section{Introduction}

An \emph{edge-colouring} of a graph $G$ is a function $\phi:E(G)\to C$, where $C$ is a set of \emph{colours}. It is \emph{proper} if incident edges receive different colours under $\phi$ and \emph{rainbow} if it is injective. 
For graphs $G$ and $H$, and an edge-colouring $\phi$ of $G$, a \emph{rainbow copy of $H$} in $G$ under $\phi$ is a subgraph $H'$ of $G$ isomorphic to $H$ such that the restriction of $\phi$ to $E(H')$ is rainbow.

Say that a graph $G$ is \emph{properly rainbow $H$-saturated} if the following two properties hold:
\begin{enumerate}
    \item There is a proper edge colouring of $G$ containing no rainbow copy of $H$;
    \item For every $e \notin E(G)$, every proper edge colouring of $G+e$ contains a rainbow copy of $H$.
\end{enumerate}
Define the \emph{proper rainbow saturation number}, denoted $\prsat(n,H)$, to be the minimum number of edges in a properly rainbow $H$-saturated graph. The study of $\prsat(n,H)$ was instigated by Bushaw, Johnston, and Rombach~\cite{Bushaw2022} inspired by previous work on the rainbow extremal number in~\cite{keevash2007rainbow} in 2007. We remark that in~\cite{Bushaw2022}, $\prsat(n,H)$ is refered to as the rainbow saturation number, but we call it proper rainbow saturation to distinguish this from the distinct usage of the term `rainbow saturation' in the literature (see, e.g., \cite{barrus2017colored}, \cite{behague2024rainbow}, \cite{girao2020rainbow}). 

In this paper, we provide many new bounds on $\prsat(n,H)$, for $H$ in various families of graphs. In particular, we reveal and utilise relationships between proper rainbow saturation and several well studied parameters such as the classical saturation number, the weak-saturation number, and the semi-saturation number of graphs, introduced below.

The saturation number, introduced by Erd\"os, Hajnal, and Moon~\cite{ErdosHajnalMoon1964} in 1964 and denoted $\sat(n,H)$, is the minimum number of edges in an $n$-vertex $H$-saturated graph $G$. That is, $G$ contains no copy of $H$, but the addition of any edge to $G$ creates a copy of $H$. The weak-saturation number, introduced by Bollob\'{a}s~\cite{Bela} in 1968 and denoted $\wsat(n,H)$, is the minimum number of edges in an $n$-vertex weakly-$H$-saturated graph $G$. That is, $G$ contains no copy of $H$ and there is an ordering $e_1,\ldots, e_m$ of the edges not present in $G$ such that for each $i$, the graph $G \cup \{e_1,\ldots,e_i\}$ contains a copy of $H$ using the edge $e_i$. These parameters have been very well studied; see the comprehensive survey of Currie, J. Faudree, R. Faudree, and Schmitt~\cite{faudree2011survey}. Note that 
\begin{equation}\label{eq:simpleinq}
    \prsat(n,H) \ge \wsat(n,H).
\end{equation}

For a graph $H$, let $\mathcal{F}^\ast(H)$ be the family of graphs that contain a rainbow copy of $H$ in all proper colourings. For a family of graphs $\mathcal{F}$, let $\sat(n,\mathcal{F})$ denote the minimum number of edges in an $n$-vertex graph $G$ containing no $F \in \mathcal{F}$ and such that adding any edge to $G$ creates some $F \in \mathcal{F}$. It follows from the definitions that 
\begin{equation}\label{obs:relation}
    \prsat(n,H) = \sat(n, \mathcal{F}^*(H)).
\end{equation}
K\'aszonyi and Tuza \cite{Kaszonyi1986} proved $\sat(n,\mathcal{F})$ is linear for all families of graphs $\mathcal{F}$. Combining this with \eqref{obs:relation}, one can deduce immediately that $\prsat(n, H) = O(n)$ for all $H$ \cite{pcbjr, pcbj}. Three of our main results, Theorems~\ref{thm:Kk_RF_bounds}, \ref{thm:Ck_upper_satast}, and \ref{thm:Kkl_upper}, are obtained by utilising \eqref{obs:relation} and analysing $\sat(n, \mathcal{F}^*(H))$ to obtain new upper bounds on $\prsat(n,H)$ for several interesting classes of graphs. 

When studying extremal parameters, one of the most natural families of graphs to consider is complete graphs. Indeed, Erd\"os, Hajnal, and Moon~\cite{ErdosHajnalMoon1964} instigated the study of saturation numbers by proving that $\sat(n,K_k) = (k-2)(n-k+2)+\binom{k-2}{2}$ whenever $2\le k \le n$. Interestingly, $\sat(n,K_k) = \wsat(n,K_k)$ and there are many beautiful proofs of this. See, for example, \cite{A,F,K84,K85,L}.

 It is not difficult to see that, as a triangle is always rainbow in a proper edge colouring, $\prsat(n,K_3) = n-1$. Bushaw, Johnston and Rombach~\cite{Bushaw2022} proved that $\prsat(n,K_k)\le O(k^4n)$ for $n$ sufficiently large. Our first application of \eqref{obs:relation} improves this to $\prsat(n,K_k) \le O(k^3n)$, for all $k \ge 4$. Interestingly, our bounds rely on results regarding the rainbow Ramsey number, discussed in Section~\ref{sub:Kk} below.

\begin{restatable}{theorem}{KkRFbounds} \label{thm:Kk_RF_bounds}
For all $k \ge 4$ and $n \ge 6$:
\begin{enumerate}
    \item[(i)] $\prsat(n,K_k) \le \left(\frac14+O\left(\frac{1}{\ln{k}}\right)\right)k^3 n.$
    \item[(ii)] $\prsat(n,K_5) \le \left\lfloor\frac{21n}{2}\right\rfloor-48.$
\end{enumerate}
\end{restatable}

We also provide an improved upper bound to that given by Theorem~\ref{thm:Kk_RF_bounds} when $k = 4$.  

\begin{restatable}{theorem}{fourcliqueconstruction} \label{thm:K4construction} For all $n \ge 7$, $\prsat(n,K_4) \le \frac{7}{2}n + O(1).$
\end{restatable}
The results on $\wsat(n,K_k)$ discussed above give that for any fixed $k$, we have $\prsat(n,K_k) \ge (k-2)n + O(1).$ There is a gap between the upper and lower bounds, but there is no reason to believe a lower bound obtained from $\wsat(n,K_k)$ should be close to the truth.

Various results have been proved for the rainbow saturation number of cycles. Improving on work of Bushaw, Johnston, and Rombach~\cite{Bushaw2022}, Halfpap, Lidický, and Masařík~\cite{halfpap2024proper} proved that $\prsat(n,C_4)=\frac{11}{6}n+o(n)$ and found upper bounds on $\prsat(n,C_5)$ and $\prsat(n,C_6)$. Here we provide the first general upper bound for cycles.

\begin{restatable}{theorem}{cycles} \label{thm:Ck_upper_satast}
For all $n\ge k \ge 7$,
\begin{enumerate}
\item[(i)] For odd $k \ge 7$, $\prsat(n,C_k) \le \frac{k-1}{2}n+\frac{1-k^2}{8}$.
\item[(ii)] For even $k \ge 10$, $\prsat(n,C_k) \le \lfloor \frac{k-1}{2}n+\frac{4-k^2}{8}\rfloor$.
\item[(iii)] $\prsat(n,C_8) \le 5n-12$.
\end{enumerate}
\end{restatable}

Another very natural family of graphs to consider is complete bipartite graphs. The saturation number for complete multipartite graphs has been determined asymptotically by Bohman, Fonoberova, and Pikhurko~\cite{BFO}, precise results for $K_{k,\ell}$ with $k,\ell \le 3$ are known~\cite{YC,GS,HLSZ}. Regarding weak saturation, the precise value of $\wsat(n,K_{k,k})$ is known for large $n$; see~\cite{K85,KMM}. The latter paper also provides asymptotically tight bounds for $\wsat(n,K_{k,\ell})$. Here we provide a general upper bound for $\prsat(n,K_{k,\ell})$.

\begin{restatable}{theorem}{completebipartite} \label{thm:Kkl_upper}
Let $\ell \ge k \ge 2$.
Then for all $n \ge k+\ell$,
\[\prsat(n,K_{k,\ell})\le \left(k-1+\frac{(k(k-1)+1)(\ell-1)}{2}\right)n-\frac{(k(k-1)^2+1)(\ell-1)+k(k-1)}{2}.\]
\end{restatable}
In fact, Theorem~\ref{thm:Kkl_upper} is a consequence of a more general result regarding the proper rainbow saturation number of bipartite graphs; see Theorem~\ref{thm:bipartite_general_upper_bound} below. As with complete graphs, the best lower bounds on $\prsat(n,K_{k,\ell})$ that we are aware of can be obtained using \eqref{eq:simpleinq} and the results from \cite{K85,KMM} on $\wsat(n,K_{k,\ell})$. For $k < \ell$ fixed, and $n$ large, we obtain $\prsat(n,K_{k,\ell}) \ge (k-1)n + O(1).$

Once again, we see that the lower bound from the weak-saturation number is not very close to the upper bound obtained via the analysis of $\sat(n,\mathcal{F}^*(H)).$

We now turn our attention to general lower bounds. As discussed in~\cite{halfpap2024proper}, it is not difficult to show that $\prsat(n,C_k) \ge n$ for $k \ge 4$. It is straightforward to see that for any graph $H$ with $\delta(H) = d$, we have $\prsat(n,H) \ge \sat(n,H) \ge \frac{d-1}{2}n$. In general, it is difficult to obtain lower bounds on $\sat(n,H)$, $\wsat(n,H)$, and $\prsat(n,H)$. Via a connection to `semi-saturation', we are able to give linear lower bounds for two large classes of graphs that may have minimum degree 1. Say a graph $G$ is {\em $H$-semi-saturated} if for all $e \in E(\overline{G})$, $G+e$ contains a copy $H'$ of $H$ such that $e \in E(H')$ (but $G$ is not necessarily $H$-free). Let $\ssat(n,H)$ be the minimum number of edges in an $H$-semi-saturated graph on $n$ vertices.

From the definitions, we see that any properly rainbow $H$-saturated graph is $H$-semi-saturated. That is, $\prsat(n,H) \ge \ssat(n,H)$.  We can obtain a general linear lower bound on $\ssat(n,H)$ and hence $\prsat(n,H)$ for some particular classes of graphs.

\begin{restatable}{theorem}{ssatthm}\label{thm:ssat_general_lower_bound}  For any $n$:
    \begin{enumerate}
        \item Any graph $H$ with no isolated edge satisfies $\prsat(n,H) \ge \ssat(n,H) \ge \lfloor \frac{n}{2} \rfloor$.
        \item If $H$ is a connected cyclic graph, then $\prsat(n,H) \ge \ssat(n,H)\ge n-1$.
    \end{enumerate}
\end{restatable}
This result is very straightforward, but we have not seen it elsewhere in the literature so we include it for completeness. We use Theorem~\ref{thm:ssat_general_lower_bound} to make progress on classifying those graphs with constant proper rainbow saturation number. The graphs with constant saturation number have been classified by K\'aszonyi and Tuza. 
\begin{theorem}[K\'aszonyi, Tuza \cite{Kaszonyi1986}] \label{thm:Kaszonyi_sat_constant}
Let $F$ be a graph.
Then $\sat(n,F) = c+ o(1)$ for some constant $c$ if and only if $F$ has an isolated edge.
\end{theorem}
With this in mind, in Section~\ref{sec:isolated} we explore whether a similar phenomenon holds for proper rainbow saturation. In particular, we prove that $\prsat(n, H \cup K_2)$ is bounded by a constant depending on $H$ whenever $H$ is a star or a matching.

In Section~\ref{sec:k4}, we prove Theorem~\ref{thm:K4construction}. The connection between $\prsat(n,H)$ and $\sat(n,\mathcal{F}^*(H))$ is explored in Section~\ref{sec:rel}, where Theorems~\ref{thm:Kk_RF_bounds}, \ref{thm:Ck_upper_satast}, and \ref{thm:Kkl_upper} are proved. Theorem~\ref{thm:ssat_general_lower_bound} is proved in Section~\ref{sec:isolated}, where we also explore the question of whether having an isolated edge is a necessary and sufficient condition for the proper rainbow saturation number to be constant. We conclude in Section~\ref{sec:conc} with some interesting open questions and directions for future research.

\section{An upper bound on $\prsat(n,K_4)$}\label{sec:k4}

Theorem~\ref{thm:K4construction} is an immediate consequence of the following, more technical, proposition.

\begin{proposition}\label{prop:K4construction}Let $n \ge 6$. Then
    $\prsat(n,K_4) \le \frac{7n}{2}+ \frac{1}{2}\left(n-4\left\lceil \frac{n-2}{4} \right\rceil \right)^2-8.$
\end{proposition}

In order to prove Proposition~\ref{prop:K4construction}, we have the following construction.
\begin{construction} \label{cns:K4}
Let $n\ge 6$.
Let $k:=\cf{n-2}{4}$ and $m:=n-4k+2$.
Define the graphs $G\cong K_2+((k-1)K_4\cup K_m)$ and $H\cong K_3+(K_3\cup K_1)$.
Then:
\begin{enumerate}
    \item[(1)] $G \notin \ca{F}^*(K_4)$, i.e., there exists a rainbow $K_4$-free proper colouring of $G$.
    \item[(2)] $G$ is $H$-saturated.
    \item[(3)] $H \in \ca{F}^*(K_4)$, i.e., every proper colouring of $H$ contains a rainbow copy of $K_4$.
\end{enumerate}
In particular, $G$ is a properly rainbow $K_4$-saturated graph on $n$ vertices.
\end{construction}
\begin{figure}
    \centering
    \includegraphics{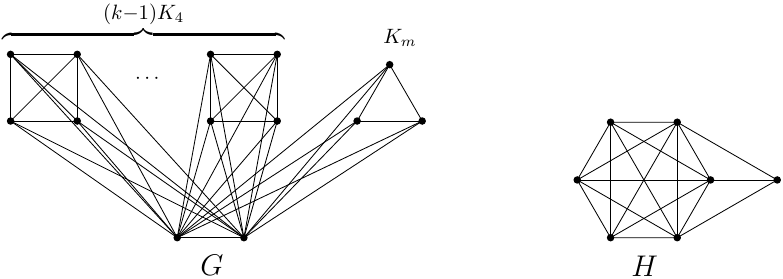}
    \caption{Graphs $G$ and $H$ from Construction~\ref{cns:K4}}
    \label{fig:k4prsat}
\end{figure}

Observe that Construction~\ref{cns:K4} implies Proposition~\ref{prop:K4construction}, since for all $n \ge 6$,
\[ \prsat(n,K_4) \le |E(G)| = \frac72 n +\frac12 (n-4k)^2-8. \]
Thus, to prove Theorem~\ref{thm:K4construction} and Proposition~\ref{prop:K4construction}, it suffices to prove Construction~\ref{cns:K4}.
\begin{proof}[Proof of Construction~\ref{cns:K4}]
We will define the graph $G$ as follows.
Let $u\neq v$ be vertices, let $A_1,\ldots,A_k$ be pairwise disjoint sets containing neither $u$ nor $v$ with $|A_i|=4$ for all $i\le k-1$ and $|A_k|=m$.
Let $V(G):=\{u,v\}\cup \bigcup_{i=1}^k A_k$ and $E(G):=\bigcup_{i=1}^k (A_k\cup\{u,v\})^{(2)}$.
For each $i\in [k]$, let $H_j:=G[A_k\cup\{u,v\}]$.
Note that we indeed have $G\cong K_2+((k-1)K_4\cup K_m)$.

We will first prove (1).
Colour the edge $uv$ with colour $c$.
Note that there is a proper $5$-colouring of $K_6$.
So let $C_1,\ldots,C_k$ be pairwise disjoint $4$-sets of colours, and for each $i \in [k]$, properly colour the edges in $H_i-uv$ with colours in $C_i \cup \{c\}$; this results in a proper colouring of $G$ because $H_i$ only shares one edge and one colour with each $H_j$ with $j\neq i$.
Any copy of $K_4$ is contained in some $H_i$, since it is missing an edge otherwise.
Each $H_i$ is coloured with at most $5$ colours, so any copy of $K_4$ in $H_i$ has at most $5<|E(K_4)|$ edges, i.e., is not rainbow.
Therefore, $G$, does not contain a rainbow copy of $K_4$, proving (1).

We will now prove (2)--(3).
Take any $xy \in E(\overline{G})$ and let $G' := G+xy$.
Without loss of generality, $x \in A_j$ and $y \in A_\ell$ for some $1\le j<\ell \le k$.
Observe that $G$ is $H$-free and
\[ H' := (G+xy)[A_j\cup \{y\}] \cong H; \]
this proves (2).

Now, colour $H'$ with a proper colouring $\phi$.
Let $C := \phi(E(H_j))$.

\textbf{Case 1:}
Suppose $|C| = 5$.
Let $S := N_{G'}(y)\cap V(H_j) = \{x,u,v\}$.
For each $a \in S$, $d_{H_j}(a) = 5$, so $a$ is incident to edges of all $5$ colours in $C$.
So the edges from $y$ to vertices in $H_j$ must have pairwise distinct colours not in $C$.
$G'[S]$ is a triangle, so it must be rainbow, with colours in $C$.
So $\{y\} \cup S$ induces a rainbow copy of $K_4$ in $G'$.

\textbf{Case 2:}
Now suppose $|C|>5$.
For each colour $c$, let $E_c := \{e \in E(H_j):\phi(e)=c\}$.
Since the colouring is proper, each colour class is a matching, so $|E_c| \le \lfloor \frac{|V(H_j)|}{2} \rfloor = 3$ for all $c \in C$.
For $i\in [3]$, let $x_i := |\{c \in C: |E_c| = i\}|$, so that $x_1+2x_2+3x_3 = |E(H_j)| = 15$.
For a colour $c$ and a copy $H'$ of $K_4$ in $H_j$, $H'$ contains two edges of colour $c$ only if $V(H') = e_1 \cup e_2$ for some $e_1,e_2 \in E_c$ with $e_1 \neq e_2$.
For each $i \in [3]$ and $c$ with $|E_c|=i$, there are $\binom{i}{2}$ such pairs.
$|C|>5$ implies $x_3<5$, so there are at most
\[\binom{1}{2}x_1+\binom{2}{2}x_2+\binom{3}{2}x_3 \le 0+\frac{15-3x_3}{2}+3x_3 = \frac{15+3x_3}{2} \le \frac{15+3(4)}{2} = \frac{27}{2}<15\]
copies of $K_4$ in $H_j$ with a repeated colour, and there are $\binom{6}{4}=15$ total copies of $K_4$ in $H_j$, so there exists a copy of $K_4$ in $H_j$ (and thus $G'$) with no repeated colour.

Therefore, in all cases, there is a rainbow copy of $K_4$ in $H'$, proving (3).
This completes the proof of Construction~\ref{cns:K4}.
\end{proof}

\section{Utilising a relationship to the saturation number}\label{sec:rel}
For a family of graphs $\mathcal{F}$ and a graph $G$, say that $G$ is \textbf{$\mathcal{F}$-saturated} if for all $F \in \mathcal{F}$, $G$ does not contain a copy of $F$, but for all $e\in E(\overline{G})$, there exists $F \in \mathcal{F}$ such that $G+e$ contains a copy of $F$.

K\'aszonyi and Tuza \cite{Kaszonyi1986} have proven that the saturation number $\sat(n,\mathcal{F})$ is linear for all families of graphs $\mathcal{F}$.
To state their result, we denote $\alpha(F)$ as the independence number of $F$ and $\ca{I}(F)$ as the set of independent sets in $F$.
They define the following parameters on the family of graphs $\mathcal{F}$:
\[
u = u(\mathcal{F}) = \mathrm{min}\{|V(F)|-\alpha(F)-1:F \in \mathcal{F}\}
\]
\[
d = d(\mathcal{F}) = \min\{|N(x)\cap S|:F \in \mathcal{F}, S \in \ca{I}(F), |S|=|V(F)|-u-1, x \in V(F)\setminus S\}
\]

We will also denote $u(F) = u(\{F\})$ and $d(F) = d(\{F\})$ for all graphs $F$.

\begin{theorem}[K\'aszonyi, Tuza \cite{Kaszonyi1986}] \label{thm:weaker_sat_general_upper_bound}
$\sat(n,\mathcal{F}) \le un+\lfloor(d-1)(n-u)/2\rfloor-\binom{u+1}{2}$ for $n$ large enough.
\end{theorem}
Theorem~\ref{thm:weaker_sat_general_upper_bound} follows from the fact that if we take the graph $G=K_u+\overline{K_{n-u}}$ and add edges $e$ to $G$ one by one such that $G+e$ is $\ca{F}$-free, we eventually obtain an $F$-saturated graph $K_u+H$ with $\Delta(H)<d$.
The bound in Theorem~\ref{thm:weaker_sat_general_upper_bound} applies to the size of any such graph $K_u+H$ when $n \ge u$.
(In fact, when $n \le u$, we have $\sat(n,\ca{F})=\binom{n}{2}$.)
Thus, we have the following.

\begin{theorem} \label{thm:weaker_sat_general_upper_bound_gen}
$\sat(n,\mathcal{F}) \le un+\lfloor(d-1)(n-u)/2\rfloor-\binom{u+1}{2}$ for $n\ge u$.
\end{theorem}

For a graph $H$, let $\mathcal{F}^\ast(H)$ be the family of graphs that contain a rainbow copy of $H$ in all proper colourings.
\eqref{obs:relation} tells us $\prsat(n,H) = \sat(n,\mathcal{F}^\ast(H))$, which means that if we identify properties of the family $\mathcal{F}^\ast(H)$, we can find bounds on $\prsat(n,H)$.

\subsection{Complete Graphs}\label{sub:Kk}
We now apply the methodology discussed above to prove Theorem~\ref{thm:Kk_RF_bounds}, restated here for convenience.  

\KkRFbounds*

To prove Theorem~\ref{thm:Kk_RF_bounds} we will exploit a relationship between $u(G)$, for $G \in \mathcal{F}^*(K_k)$ and the \emph{rainbow Ramsey number}, introduced by Jamison, Jiang, and Ling \cite{Jamison2003}. For graphs $G_1$ and $G_2$, let $R^*(G_1,G_2)$, the \emph{rainbow Ramsey number}, be the minimum integer $n$ (if it exists) such that any edge-colouring of $K_n$ must contain either a monochromatic copy of $G_1$ or a rainbow copy of $G_2$.
Jamison, Jiang, and Ling \cite{Jamison2003} determined that $R^*(G_1,G_2)$ exists if and only if $G_1$ is a star or $G_2$ is a forest. In particular, for all graphs $H$, $R^*(P_3,H)$ exists. In the proof of Theorem~\ref{thm:Kk_RF_bounds}, we will apply the following result on $R^*(P_3,K_k)$.

\begin{theorem}[Alon, Jiang, Miller, Pritikin \cite{Alon2003}] \label{thm:Alon_rainbow_ramsey_complete}
For $k \ge 3$, $R^*(P_3,K_k) = \Theta\left(\frac{k^3}{\ln{k}}\right)$.
\end{theorem}

We are now prepared to prove Theorem~\ref{thm:Kk_RF_bounds}.

\begin{proof}[Proof of Theorem~\ref{thm:Kk_RF_bounds}]
Let $r = R^*(P_3,K_{k-1})$, $u = u(\mathcal{F}^\ast(K_k))$, $d = d(\mathcal{F}^\ast(K_k))$, and $\ell = \binom{k-1}{2}(k-3)-r+k$.
We will show that $u=r-1$ and $d\le (k-3)\binom{k-1}{2}-r+k$.
Let $H \cong K_r$, let $J \cong \overline{K_\ell}$, and let $G = H+J$.
\begin{claim*}
$G \in \ca{F}^\ast(K_k)$ and $u \le u(G) = r-1$.
\end{claim*}
\begin{proof}[Proof of Claim]
Let $G$ be coloured with any proper colouring.
Because $H \cong K_r$, $H$ contains a rainbow copy $F$ of $K_{k-1}$.

For each colour $c$ present in $F$, exactly one edge $uv\in E(F)$ receives colour $c$, and each $w\in V(F)\setminus\{u,v\}$ is incident to at most one edge of colour $c$, so there are $\le k-3$ vertices $y\in V(G)\setminus V(F)$ such that $G[V(F)\cup\{y\}]$ has the repeated colour $c$.
Thus, since $|V(G)\setminus V(F)| = r+\ell-k >\binom{k-1}{2}(k-3)$, there exists $y \in V(G)\setminus V(F)$ such that $G[V(F)\cup\{y\}]$ is rainbow.

So $G \in \mathcal{F}^\ast(K_k)$.
Now $\alpha(G) = \ell$, so $u \le u(G) = |V(G)|-\ell-1 = r-1$, as desired.
\end{proof}

\begin{claim*}
$u \ge r-1$.
\end{claim*}
\begin{proof}[Proof of Claim]
Let $G'$ be a graph with $u(G') \le r-2$, let $I\subseteq V(G')$ be an independent set with $|I| = |V(G')|-u(G')-1 \ge |V(G')|-r+1$, and let $X' = V(G')\setminus I$.
Because $|X'|< r$, $X'$ can be properly coloured with no rainbow copy of $K_{k-1}$.
Any copy of $K_k$ with at least $k-1$ vertices in $X'$ is thus not rainbow, and any set of $k$ vertices with at least two vertices in $I$ is missing an edge.
So $G' \notin \mathcal{F}^\ast(K_k)$.
Therefore, $u \ge r-1$, as desired.
\end{proof}

By the above two claims, $u = r-1$.
$V(J)$ is the only independent set of $G$ of size $|V(G)|-u-1$, and any $x \in V(G)\setminus V(J) = V(H)$ has $|N(x)\cap V(J)|\le |V(J)| = \ell$.
$d$ is the minimum possible $|N(x) \cap S|$ where $S$ is an independent set of some $F \in \mathcal{F}^\ast(K_k)$ of size $|V(F)|-u-1$ and $x \in V(F)\setminus S$, so $d \le \ell = \binom{k-1}{2}(k-3)-r+k$.

We will now prove (1) and (2).

(i)
By Theorem~\ref{thm:Alon_rainbow_ramsey_complete}, $r=\Theta(k^3n/\ln{k})$, so by Theorem~\ref{thm:weaker_sat_general_upper_bound} and \eqref{obs:relation}, for all $n$ large enough,
$$\prsat(n,K_k)\le un+\frac12 dn 
    \le rn + \frac14 k^3 n 
    = \left(\frac14 k^3 + O\left( \frac{k^3}{\ln{k}} \right) \right) n.$$

(ii)
Consider $k=5$. We apply the following simple claim.
\begin{claim} \label{prop:K4_RR}
$R^*(P_3,K_4)=7$.
\end{claim}
\begin{proof}
Consider the graphs $G$ and $H$ from Construction~\ref{cns:K4}.
By property (1) of Construction~\ref{cns:K4}, $G$, which contains a copy of $K_6$, can be properly coloured with no rainbow copy of $K_4$.
In Construction~\ref{cns:K4}, property (3) shows that the graph $H$ is a graph of order $7$ that contains a rainbow copy of $K_4$ in all proper colourings.
Therefore, $R^*(P_3,K_4) = 7$.
\end{proof}

By Claim~\ref{prop:K4_RR}, $r=R^*(P_3,K_4)=7$, so $u=6$.
Also, $d \le (5-3)\binom{5-1}{2}-7+4 = 9$.
So by Theorem \ref{thm:weaker_sat_general_upper_bound_gen} and \eqref{obs:relation}, for all $n\ge u = 6$,
\[ \prsat(n,K_5) = \sat(n,\mathcal{F}^\ast(K_5))  \le 6n+\left\lfloor\frac{9(n-6)}{2}\right\rfloor-\binom{7}{2} 
    = \left\lfloor\frac{21n}{2}\right\rfloor-48. \qedhere \]
\end{proof}

\subsection{Cycles}

\cycles*

\begin{proof}
Let $\ell := \lceil k/2\rceil$.
Let $u:=u(\ca{F}^*(C_k))$ and $d(\ca{F}^*(C_k))$.
We will show that $u=\ell-1$ and $d\le z$, where
\[z:= \begin{cases}
    1, & k \in \{7,9,11,\ldots\} \\
    2, & k \in \{10,12,14,\ldots\} \\
    5, & k =8.
\end{cases} \]
Let $X:= \{x_1,\ldots,x_\ell\}$ and $Y$ be a set disjoint from $X$ with $|Y|$ large (we can take $|Y| \ge 5\ell$).
Let $H$ be the complete graph on $X$, and let $I$ be the empty graph on $Y$.
Let $Y':=\{y_1'\ldots,y_z'\}$ be a $z$-subset of $Y$.
Let $G:= (H+I)-\{x_2y:y \in Y\setminus Y'\}$.

\begin{claim} \label{clm:RF(Ck)_has_rainbow_Ck}
$G \in \ca{F}^\ast(C_k)$.
\end{claim}
\begin{proof}[Proof of Claim]
Let $\phi$ be a proper colouring of $G$.
We will begin by defining a rainbow path $p_2$ of length $\le 4$ in each of the following cases.

\textbf{Case 1:}
$k=8$.
Note that $p_1:= x_1\bc y_1'\bc x_2$ is a length-$2$ path and is thus rainbow.
For each colour $c$ present in $p_1$, since $\phi$ is proper, there are at most $2$ vertices in $\{y_2',y_3',y_4',y_5'\}$ whose edges to $\{x_2,x_3\}$ receive colour $c$; there is at most $1$ such vertex if $c=\phi(y_1'x_2)$.
So there are at most $3$ vertices $y$ in $Y$ such that $p_1,y,x_3$ is not rainbow.
Thus, since $|\{y_2',y_3',y_4',y_5'\}|=4$, there exists $y_2 \in \{y_2',y_3',y_4,y_5'\}$ such that $p_2 := p_1,y_2,x_3$ is rainbow.

\textbf{Case 2:}
$k \in \{7,9,11,\ldots\}$.
Since edges incident with $y_1'$ receive distinct colours, there is at most one vertex $x \in \{x_3,\ldots,x_\ell\}$ such that $\phi(x_1x_2) = \phi(y_1'x)$.
So because $|\{x_3,\ldots,x_\ell\}|\ge \ell-2>1$, there exists $x \in \{x_3,\ldots,x_\ell\}$ such that $p_2 := x_1,x_2,y_1',x$ is rainbow.
Without loss of generality, say $x=x_3$, and let $y_1 := y_1'$.

\textbf{Case 3:}
$k \in \{10,12,14,\ldots\}$.
Note that $p_1 := y_1'\bc x_2\bc y_2'$ is a length-$2$ path and is thus rainbow.
Since $\phi$ is proper, there is at most one vertex $x \in X\setminus\{x_2\}$ such that $\phi(xy_1') = \phi(x_2y_2')$, and $|X\setminus\{x_2\}| \ge 4 >1$, so suppose without loss of generality that $\phi(x_1y_1') \neq \phi(x_2y_2')$.
Now there are at most two vertices $x \in X\setminus\{x_1,x_2\}$ such that $\phi(y_2'x) \in \{\phi(x_1y_1'),\phi(y_1'x_2)\}$, and $|X\setminus\{x_1,x_2\}|\ge 3 > 2$, so suppose without loss of generality that $\phi(y_2'x_3) \notin \{\phi(x_1y_1'),\phi(y_1'x_2)\}$.
Then $p_2 := x_1,p_1,x_3$ is rainbow.

In all cases, we have defined the rainbow path $p_2=x_1\bc \ldots\bc x_3$ ($p_2$ has length $4$ when $k$ is even, and it has length $3$ when $k$ is odd).
We will now extend $p_2$ to a rainbow cycle of length $k$.

Take any $i \in \{3,\ldots,\ell\}$ and suppose that the rainbow path $p_{i-1} := p_2,y_3,x_4,\ldots,y_{i-1},x_i$ of length $\le 2(i-1)$ has already been defined.
There are at most $2(i-1)$ colours in $p_{i-1}$.
For each colour $c$ present in $p_{i-1}$, there are at most $2$ vertices in $Y \setminus V(p_{i-1})$ whose edges to $\{x_i,x_{{i+1}\, (\mr{mod}\, \ell)}\}$ receive colour $c$; there is at most $1$ such vertex if $c$ is the colour of $y_{i-1}x_i$.
So there are at most $4(i-1)-1 = 4i-5$ vertices $y$ in $Y$ such that $p_{i-1},y,x_{i+1\, (\mr{mod}\, \ell)}$ is not rainbow.

Now, $|Y\setminus V(p_{i-1})|>4i-5$, so there exists $y_i \in (N(x_i)\cap N(x_{i+1\, (\mr{mod}\, \ell)})) \setminus \{y_1,\ldots,y_{i-1}\}$ such that $p_i:= p_{i-1},y,x_{i+1\, (\mr{mod}\, \ell)}$ is rainbow.
Applying this argument iteratively yields a rainbow cycle of length $k$.
Therefore, $G$ contains a rainbow copy of $C_k$ in all proper colourings, so $G \in \ca{F}^\ast(C_k)$.
This completes the proof of Claim~\ref{clm:RF(Ck)_has_rainbow_Ck}.
\end{proof}

Since any independent set $I$ in $G$ with more than one element cannot contain a universal vertex in $X\setminus\{x_2\}$, and $x\in I$ implies $N(x) \cap I = \emptyset$, we have $|I| \le \max\{|Y|,|\{x_2\}\cup (Y\setminus Y')|\} = |Y|$.
So $\alpha(G) = |Y|$, and so $u(G) = |X|-1=\ell-1$.
By Claim \ref{clm:RF(Ck)_has_rainbow_Ck}, $u \le u(G) = \ell-1$.

Consider the graph $C_k$ with cycle $v_1,\ldots,v_k,v_1$.
An independent set in $C_k$ cannot contain two consecutive vertices $v_i,v_{i+1}$, so $\{v_1,v_3,\ldots,v_{\lfloor k/2 \rfloor}\}$ is a maximal independent set by cardinality.
Thus, $\alpha(C_k) = \lfloor k/2 \rfloor$.
Any graph $F$ with $u(F) < \ell-1$ has $\alpha(F) > |V(F)|-\ell$, so letting $I$ be an independent set in $F$ with $|I|=\alpha(F)$, any copy $C$ of $C_k$ in $F$ would have 
\[|V(C)\cap I| = |V(C)\setminus (V(F)\setminus I)| > k-\ell = \left\lfloor \frac{k}{2}\right\rfloor, \]
contradicting $\alpha(C_k) = \lfloor k/2 \rfloor$.
So $F$ is $C_k$-free, and so $F \notin \ca{F}^\ast(C_k)$.
Thus, $u\ge \ell-1$.
Therefore, $u = \ell-1$.

In the graph $G$, $Y$ is an independent set of size $|V(G)|-u-1$ and $x_2 \notin Y$, so $d(\ca{F}^\ast(C_k)) \le |N(x_2)\cap Y| = z$.
Therefore, by Theorem~\ref{thm:weaker_sat_general_upper_bound_gen}, for all $n \ge k (\ge u)$,
\[ \prsat(n,C_k)\le un+\ff{(d-1)(n-u)}{2}-\binom{u+1}{2} \le \begin{cases}
    5n-12, & k=8 \\
    \frac{k-1}{2}n+\frac{1-k^2}{8}, & k\ge 7 \text{ odd} \\
    \lfloor\frac{k-1}{2}n+\frac{4-k^2}{8}\rfloor, & k \ge 10 \text{ even}.
\end{cases} \]
as desired.
\end{proof}

\begin{remark}
    We note that the method used in the proof of Theorem~\ref{thm:Ck_upper_satast} can reprove the bound on $\prsat(n,C_5)$ obtained in~\cite{halfpap2024proper}. However, the bound we achieve for $C_6$ is weaker than their bound. 
    Note in addition that we begin the proof by showing that $d \le z$. In fact, $d = z$, but we omit the proof as we do not require this fact. 
\end{remark}

\subsection{Bipartite Graphs}

\begin{restatable}{theorem}{bipartite} \label{thm:bipartite_general_upper_bound}
Let $H$ be a bipartite graph with bipartition $(X,Y)$, and let $k=|X|$ and $\ell = |Y|$.
Suppose $\ell = \alpha(H)$.
Then for all $n \ge k+\ell$,
\[\prsat(n,H)\le \left(k-1+\frac{(k(k-1)+1)(\ell-1)}{2}\right)n-\frac{(k(k-1)^2+1)(\ell-1)+k(k-1)}{2}.\]
\end{restatable}

\begin{proof}
Let $G$ be the complete bipartite graph with parts $A,B$, where $A$ and $B$ are sets with $|A|=k$ and $|B| = k(k-1)(\ell-1)+\ell$.

\begin{claim} \label{clm:bipartite_u(F(H))_upper}
$G \in \ca{F}^\ast(H)$ and $u(\ca{F}^\ast(H)) \le u(G) = k-1 = u(H)$.
\end{claim}
\begin{proof}[Proof of Claim]
The fact $k-1 =u(H)$ follows from the hypothesis that $\ell=\alpha(H)$.

We also have $u(G)=|V(G)|-\alpha(G)-1 = k-1$.
So it suffices to prove that $u(\ca{F}^\ast(H)) \le u(G)$.

Let $G$ be coloured with any proper colouring.
We will iteratively define a sequence of sets $\emptyset = B_0 \subseteq B_1 \subseteq \cdots \subseteq B_{\ell} \subseteq B$ such that for all $i$, $|B_i| = i$ and $G[A\cup B_i]$ is rainbow.
$G[A\cup B_0]$ is rainbow, since it has no edges.

Take any $i \in [\ell]$ and suppose $B_{i-1}\subseteq B$ has already been defined so that $|B_{i-1}|=i-1$ and $G[A\cup B_{i-1}]$ is rainbow.

For each $x \in A$, $G[A\cup B_{i-1}]$ has exactly $(k-1)(i-1)$ distinct colours on edges not incident with $x$.
Thus, for each $x \in A$, there are at most $(k-1)(i-1)$ vertices $y$ in $B\setminus B_{i-1}$ such that $xy$ receives a colour in $G[A\cup B_{i-1}]$.
Thus, there are at most $k(k-1)(i-1)$ vertices $y$ in $B\setminus B_{i-1}$ such that $G[A\cup B_{i-1}\cup \{y\}]$ is not rainbow.
Since 
\[|B\setminus B_{i-1}| \ge k(k-1)(\ell-1)+\ell-(\ell-1)> k(k-1)(\ell-1),\]
there exists a vertex $y_i \in B\setminus B_{i-1}$ such that $G[X\cup B_{i-1}\cup \{y_i\}]$ is rainbow.
Set $B_i := B_{i-1} \cup \{y_i\}$.

By repeating this definition iteratively for all $i \in [\ell]$, we obtain a set $B_{\ell} \subseteq B$ such that $G[A\cup B_{\ell}]$ is a rainbow copy of $K_{k,\ell}$ in $G$.
This contains a rainbow copy of $H$.
Therefore, $G \in \ca{F}^\ast(H)$, so $u(\ca{F}^\ast(H)) \le u(G)$.
\end{proof}

\begin{claim} \label{clm:bipartite_u(F(H))_lower}
$u(\ca{F}^\ast(H)) \ge k-1$.
\end{claim}
\begin{proof}
Let $G'$ be a graph with $u(G')<k-1$, let $I$ be an independent set in $G'$ with $|I| = |V(G')|-u(G')-1$, and let $J := V(G')\setminus I$.
Then $|J|<k$.
If $G'$ contains a copy $F$ of $H$, then $V(F)\cap I$ is an independent set in $F$ with $|V(F)\cap I| = |V(F)\setminus J| > \ell$, contradicting the hypothesis that $\alpha(F) = \ell$. Thus, $G'$ does not contain a copy of $H$.

This proves that any graph $G'$ with $u(G')<k-1$ is $H$-free (and is thus rainbow $H$-free in any proper colouring), so $u(\ca{F}^\ast(H)) \ge k-1$.
\end{proof}
Combining Claims \ref{clm:bipartite_u(F(H))_upper} and \ref{clm:bipartite_u(F(H))_lower}, we obtain $u(\ca{F}^\ast(H)) = u(H)=k-1$.
Now $B$ is an independent set in $G$ with $k(k-1)(\ell-1)+\ell=|V(G)|-u(\ca{F}^\ast(H))-1$ elements, and every vertex $x$ in $V(G)\setminus B = A$ has $|N(x) \cap B| = |B|$, so 
\[d(\ca{F}^\ast(H)) \le |B| = k(k-1)(\ell-1)+\ell.\]
Thus, by Theorem \ref{thm:weaker_sat_general_upper_bound_gen}, for all $n \ge k+\ell \ (> k-1 = u(\ca{F}^\ast(H)))$,
\begin{align*}
    \prsat(n,H) &= \mr{sat}(n,\ca{F}^\ast(H)) \\
    &\le (k-1)n+\left\lfloor\frac{((k(k-1)(\ell-1)+\ell-1)(n-k+1)}{2}\right\rfloor-\binom{k}{2} \\
    &\le \left(k-1+\frac{(k(k-1)+1)(\ell-1)}{2}\right)n-\frac{(k(k-1)^2+1)(\ell-1)+k(k-1)}{2},
\end{align*}
as desired.
\end{proof}

A simple application of Theorem~\ref{thm:bipartite_general_upper_bound} yields our bound on $\prsat(n,K_{s,t})$, restated here for convenience.

\completebipartite*
\begin{proof}
Let $K_{k,\ell}$ be the complete bipartite graph with parts $X,Y$, where $|X|=k$ and $|Y|=\ell$.
Note that any independent set in $K_{k,\ell}$ is contained entirely in either $X$ or $Y$, since otherwise it contains some $x \in X$ and $y\in Y$ and thus contains the edge $xy \in E(K_{k,\ell})$.
So because $X$ and $Y$ are independent sets, $\alpha(K_{k,\ell}) = \max\{k,\ell\}= \ell$.
Therefore, the result follows from Theorem \ref{thm:bipartite_general_upper_bound}.
\end{proof}

\section{Constant proper rainbow saturation number} \label{sec:isolated}

In this section we explore the family of graphs with constant proper rainbow saturation number. We will see via a connection to semi-saturation that any such graph necessarily has an isolated edge.

\subsection{Lower Bounds from Semi-Saturation}
Here we provide the (simple) proof of Theorem~\ref{thm:ssat_general_lower_bound}, restated here for convenience. 

\ssatthm*
\begin{proof}

For (i), let $G$ be a $H$-semi-saturated graph on $n$ vertices.  If a graph $G$ has two isolated vertices $v,u$, then since $H$ has no isolated edge, $G+vu$ does not contain a copy of $H$ that contains the edge $vu$, contradicting $G$ being $H$-semi-saturated.
 So $G$ contains at most one isolated vertex. Therefore, $|E(G)| = \frac{\sum_{v\in V(G)}d(v)}{2} \ge \frac{n-1}{2}$ and $|E(G)|$ is an integer, so $|E(G)| \ge \lfloor \frac{n}{2} \rfloor$ as required for (i).

For (ii), let $G$ be an $H$-semi-saturated graph.
If $G$ has two distinct tree components $T_1,T_2$, then letting $u\in V(T_1)$ and $v\in V(T_2)$, the component of $G+uv$ containing $uv$ is a tree and is thus $H$-free.
This contradicts $G$ being $H$-semi-saturated.
So $G$ has at most one tree component.
Letting $T$ be a component of $G$ such that $|E(T)|-|V(T)|$ is minimal, all components of $G-V(T)$ are cyclic, so we have
\[ |E(G)| = |E(G-V(T))|+|E(T)|\ge n-|V(T)|+(|V(T)|-1)=n-1, \]
as required for (ii).
\end{proof}
\subsection{Graphs with an Isolated Edge}
Theorem \ref{thm:ssat_general_lower_bound} shows that $\prsat(n,H)$ and $\mathrm{ssat}(n,H)$ are linear in $n$ if $H$ has no isolated edge.
For $\ssat(n,H)$, the converse is also true by the following theorem by K\'aszonyi and Tuza. 
\begin{corollary}[K\'aszonyi, Tuza \cite{Kaszonyi1986}] \label{cor:Kaszonyi_sat_constant}
Let $F$ be a graph.
Then $\sat(n,F) = c+ o(1)$ for some constant $c$ if and only if $F$ has an isolated edge.
\end{corollary}

Because $\mathrm{ssat}(n,H) \le \sat(n,H)$, we obtain the following consequence of Corollary~\ref{cor:Kaszonyi_sat_constant} and Theorem~\ref{thm:ssat_general_lower_bound}.
\begin{proposition}
$\mathrm{ssat}(n,H) = O(1)$ if and only if $H$ has an isolated edge.
\end{proposition}
Is it natural to wonder whether $\prsat(n,H) = O(1)$ whenever $H$ has an isolated edge. We show that this is indeed the case for the graphs $K_3 \cup K_2$, $K_{1,k} \cup K_2$, and $mK_2$ for $k,m \in \mathbb{N}$.

\begin{restatable}{theorem}{starwithedge} \label{thm:star_and_K2}
For all $k \ge 1$ and $n \ge 2\cf{k}{2}+2$, $\prsat(n,K_{1,k}\cup K_2)\le \binom{2\cf{k}{2}+2}{2}$.
\end{restatable}

\begin{proof}
Let $\ell:= 2\cf{k}{2}+2$, and note that $k \in \{\ell-3,\ell-2\}$.
Let $H \cong K_\ell$ and $G = H \cup \overline{K_{n-\ell}}$.
\begin{claim}
    $G$ has a rainbow $K_{1,\ell-3}\cup K_2$-free proper colouring.
\end{claim}
\begin{proof}[Proof of Claim]
Let $\phi$ be a proper $(\ell-1)$-colouring of $H$, and note that every colour class in $\phi$ is a perfect matching.
Take any copy $F$ of $K_{1,\ell-3}\cup K_2$ in $H$.
Let $u$ be a vertex in $F$ of maximum degree, and let $vw$ be the isolated edge in $F$ that does not contain $u$.
Since $|V(F)| = |V(H)|$, the component of $F$ containing $u$ contains every edge incident with $u$ except $uv$ and $uw$.
Since every colour class in $\phi$ is a perfect matching, there exists an edge $e \in E(F)$ incident with $u$ in colour $\phi(vw)$, so $F$ is not rainbow.
Thus, $\phi$ is rainbow $K_{1,\ell-3}\cup K_2$-free.
\end{proof}

\begin{claim}
For any $vu \in E(\overline{G})$, $G+vu$ contains a rainbow copy of $K_{1,\ell-2}\cup K_2$ in any proper colouring.
\end{claim}
\begin{proof}[Proof of Claim]
Let $vu \in E(\overline{G})$ such that $v \notin V(H)$, and let $\phi$ be a proper colouring of $G+vu$.
Let $A$ be an $(\ell-1)$-subset of $V(H)\setminus\{u\}$, let $H' = H[A]$, and let $E' = \{e \in E(H'): \phi(e) = \phi(vu)\}$.
$E'$ is a matching, and $|V(H')| = \ell-1$, which is odd, so $E'$ is a non-perfect matching.
So there exists $x \in V(H')$ such that $\{xy:y \in V(H') \setminus \{x\}\} \cap E' = \emptyset$.
Let $S$ be the spanning star subgraph of $H'$ with universal vertex $x$, and note that $S$ is a rainbow copy of $K_{1,\ell-2}$ and contains no edge of colour $\phi(vu)$.
So the graph $F$ with $V(F) = V(S) \cup \{v,u\}$ and $E(F) = E(S) \cup \{vu\}$ is a rainbow copy of $K_{1,\ell-2} \cup K_2$ in $G$.
\end{proof}

By the above two claims, since $K_{1,\ell-3}\cup K_2 \subset K_{1,k}\cup K_2 \subset K_{1,\ell-2} \cup K_2$, $G$ is properly rainbow $K_{1,k}\cup K_2$-saturated.
Therefore, 
\[\prsat(n,K_{1,k}\cup K_2) \le |E(G)| = \binom{\ell}{2},\]
as desired.
\end{proof}

\begin{restatable}{theorem}{trianglewithedge} \label{thm:K3UK2}
For all $n \ge 6$, $\prsat(n,K_3 \cup K_2) \le 15$, and the graph $K_6 \cup \overline{K_{n-6}}$ is properly rainbow $K_3 \cup K_2$-saturated.
\end{restatable}

\begin{proof}
Note that $K_2$ and $K_3$ are rainbow in any proper colouring, so a copy of $K_3 \cup K_2$ is rainbow if and only if the colour of the $K_2$ does not also colour an edge in the $K_3$.

Let $H \cong K_6$ and $G = H \cup \overline{K_{n-6}}$.

\begin{claim} \label{clm:K3UK2_no_rainbow}
Let $\phi$ be a proper colouring of $G$, and for each colour $c$, let $E_c = \{e\in E(G): \phi(e) = c\}$.
Then the following are equivalent:
\begin{enumerate}
    \item $\phi$ is a $5$-colouring.
    \item For any colour $c$, $E_c$ is a perfect matching of $H$.
    \item $\phi$ is rainbow $K_3\cup K_2$-free.
\end{enumerate}
\end{claim}
\begin{proof}[Proof of Claim]
$(1) \iff (2)$:
Clear.

$(2) \implies (3)$:
Suppose that $E_c$ is a perfect matching of $H$ for every colour $c$.
Let $F$ be a copy of $K_3$ in $G$, and let $J := G[V(H) \setminus V(F)]$.
For each $c \in \phi(E(F))$, $E_c \cap E(J)\neq \emptyset$, and $J \cong K_3$, so $\phi(E(F)) = \phi(E(J))$.
Thus, for any edge $e \in E(G)$ which is vertex-disjoint from $F$, we have $|\phi(E(F)\cup \{e\})| = |\phi(E(F))| = 3$.
So $\phi$ is rainbow $K_3 \cup K_2$-free.

$(3) \implies (2)$:
To prove the contrapositive, suppose that there exists a colour $c$ such that $E_c$ is not a perfect matching of $H$, i.e., $1 \le |E_c| < 3$.
Then $E_c$ is a matching on at most $2$ edges and at most $4$ vertices, so there exist $u,v \in V(H)$ such that $uN(u) \cap E_c = vN(v) \cap E_c = \emptyset$.
Let $xy \in E_c$ and $w \in V(H) \setminus \{v,u,x,y\}$.
Then since no edge in $G[\{v,u,w\}]$ receives colour $c$ and $\phi(xy) = c$, $G[\{v,u,w\}] \cup G[\{x,y\}]$ is a rainbow copy of $K_3 \cup K_2$ in $G$, as desired.
\end{proof}

\begin{claim} \label{clm:K3UK2_has_rainbow}
For any $e \in E(\overline{G})$, $G+e$ contains a rainbow copy of $K_3 \cup K_2$ in any proper colouring.
\end{claim}
\begin{proof}[Proof of Claim]
Since $H$ has no non-edges, $e = uv$ for some $u, v \in V(G)$ with $u \notin V(H)$.

Suppose for contradiction that there exists a rainbow $K_3\cup K_2$-free proper colouring $\phi$ of $G+e$.
Let $C = \phi(E(G))$ and $c' = \phi(uv)$.

\textbf{Case 1:}
$v \in V(H)$. We have that $\phi|_{E(G)}$ is a proper $5$-colouring of $H$ by Claim \ref{clm:K3UK2_no_rainbow}, and $d_H(v)=5$.
So for all $c \in C$, there exists an edge incident with $v$ of colour $c$, so $c' \notin C$.
So for any copy $F$ of $K_3$ in $H-v$, $F \cup G[\{u,v\}]$ is a rainbow copy of $K_3 \cup K_2$ in $G+uv$, a contradiction.

\textbf{Case 2:}
$v \notin V(H)$.
Since $E_{c'} \cap E(H)$ is a matching on at most $3$ edges, let $M =\{v_1v_2,v_3v_4,v_5v_6\}$ be a matching in $H$ such that $E_{c'}\cap E(H) \subseteq M$.
Then the subgraph $F = G[\{v_1,v_3,v_5\}]$ of $G+uv$ does not contain an edge of colour $c'$, so $F\cup G[\{u,v\}]$ is a rainbow copy of $K_3 \cup K_2$ in $G+uv$, a contradiction.

Both cases lead to a contradiction.
Therefore, $G+e$ contains a rainbow copy of $K_3 \cup K_2$ in any proper colouring.
This completes the proof of Claim~\ref{clm:K3UK2_has_rainbow}.
\end{proof}
Note that a proper $5$-colouring of $G$ exists.
So by Claim \ref{clm:K3UK2_no_rainbow}, a rainbow $K_3 \cup K_2$-free proper colouring of $G$ exists.
Thus, together with Claim \ref{clm:K3UK2_has_rainbow}, this implies that $G$ is properly rainbow $K_3\cup K_2$-saturated.
Therefore, $\prsat(n,K_3\cup K_2) \le |E(G)| = \binom{6}{2} = 15$.
\end{proof}

\begin{restatable}{theorem}{matching} \label{thm:mK2}
For all $m \ge 3$ and $n \ge m^2$, $\prsat(n,mK_2) \le 3\binom{m}{2}$.
\end{restatable}

\begin{proof}
\begin{figure}
    \centering
    \includegraphics{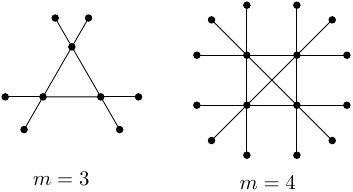}
    \caption{Graph $H$ from Theorem \ref{thm:mK2} in the cases $m=3$ (left) and $m=4$ (right). The graph $G = H \cup \overline{K_{n-m^2}}$ is properly rainbow $mK_2$-saturated.}
    \label{fig:mK2}
\end{figure}
Let $X:=\{x_1,\ldots,x_m\}$, and for each $i \in [m]$, let $Y_i:=\{y_{i,1},\ldots,y_{i,m-1}\}$, where all $x_i$ and $y_{i,j}$ are pairwise distinct.
Let $Y:=\bigcup_{i=1}^m Y_i$, let $E_1:= X^{(2)}$, and let $E_2 := \bigcup_{i=1}^m x_i Y_i$.
Let $H$ be the graph with $V(H)=X\cup Y$ and $E(H) = E_1 \cup E_2$, and let $G:= H \cup \overline{K_{n-m^2}}$.
\begin{claim} \label{clm:mK2_free}
There exists a rainbow $mK_2$-free proper colouring of $G$.
\end{claim}
\begin{proof}[Proof of Claim]
Let $\phi$ be a proper colouring of $G$ with $\phi(x_iy_{i,j})=j$ for all $(i,j) \in [m]\times[m-1]$.
Since $Y$ is an independent set in $G$, every $m$-edge matching in $G$ is in $E_2$, which only receives $m-1$ colours.
So $\phi$ is rainbow $mK_2$-free.
\end{proof}
\begin{claim} \label{clm:mK2_colouring}
Let $\phi$ be a proper colouring of $G$.
If $\phi$ is rainbow $mK_2$-free, then $|\phi(E_2)|=m-1$.
\end{claim}
\begin{proof}[Proof of Claim]
To prove the contrapositive, suppose $\phi$ is a proper colouring of $G$ with $|\phi(E_2)|\ge m$.
Let $c \in \phi(E_2)$ such that the colour class $E_c$ of $c$ has $|E_c\cap E_2|<m$.
Let $k,\ell \in [m]$ so that $E_c \cap x_kY_k\neq \emptyset = E_c \cap x_\ell Y_\ell$; without loss of generality, $k=1$ and $\ell=m$.
Let $e_1 \in E_c \cap x_1Y_1$.
For each $i \in [2,m-1]$, let $e_i \in x_iY_i$ such that $\phi(e_i) \notin \phi(\{e_1,\ldots,e_{i-1}\})$.
Let $e_m \in x_mY_m$ such that $\phi(e_m) \notin \phi(\{e_2,\ldots,e_{m-1}\})$.
Then $\{e_1,\ldots,e_m\}$ is an $m$-edge rainbow matching, so $\phi$ is not rainbow $mK_2$-free, proving the claim.
\end{proof}
\begin{claim} \label{clm:has_mK2}
For all $e \in E(\overline{G})$, $G+e$ contains a rainbow copy of $mK_2$ in any proper colouring.
\end{claim}
\begin{proof}[Proof of Claim]
Let $e=e_1 \in E(\overline{G})$, and suppose there exists a rainbow $mK_2$-free proper colouring $\phi$ of $G+e$.
By Claim~\ref{clm:mK2_colouring}, $|\phi(E_2)|=m-1$.
Since $e \not\subseteq X$, suppose without loss of generality that $e \cap X \subseteq \{x_1\}$ and $|e\cap (\{x_1\}\cup Y_1)| \ge |e \cap Y_2| \ge |e\cap Y_3|=\cdots = |e\cap Y_m|=0$.

\textbf{Case 1:}
$\phi(e)\notin \phi(E_2)$.
Let $e_2 \in x_2(Y_2\setminus e)$.
For all $i \in [3,m]$, let $e_i \in x_iY_i$ such that $\phi(e_i)\notin \phi(\{e_1,\ldots,e_{i-1}\})$.
Then $\{e_1,\ldots,e_m\}$ is a rainbow $m$-edge matching.

\textbf{Case 2:}
$\phi(e) \in \phi(E_2)$.
Note that $e\cap X = \emptyset$.
For all $i \in [3,m]$, let $e_i \in x_i Y_i$ such that $\phi(e_i)\notin \phi(\{e,e_3,e_4,\ldots,e_{i-1}\})$.
Observe that $\phi(\{e,e_3,e_4,\ldots,e_m\}) \subseteq \phi(E_2)$ and $\phi(x_1x_2) \notin \phi(x_1Y_1) = \phi(E_2)$.
Thus, $\{e,x_1x_2,e_3,e_4,\ldots,e_m\}$ is a rainbow $m$-edge matching.

In both cases, $G+e$ contains a rainbow copy of $mK_2$, which is a contradiction.
This completes the proof of Claim~\ref{clm:has_mK2}.
\end{proof}
By Claims~\ref{clm:mK2_free} and~\ref{clm:has_mK2}, $G$ is properly rainbow $mK_2$-saturated, so $\prsat(n,mK_2)\le |E(G)| = 3\binom{m}{2}$, as desired.
\end{proof}

\section{Conclusion}\label{sec:conc}

In Section~\ref{sec:rel} we provide new upper bounds on $\prsat(n,H)$, when $H$ is a clique, a cycle, or a complete bipartite graph.
The method used essentially bounds $\prsat(n,H)$ by the number of edges in a properly rainbow $H$-saturated graph with as many universal vertices as possible, so it is not necessarily close to the true value of $\prsat(n,H)$.

\begin{problem}
    Improve the bounds on $\prsat(n,H)$ when $H$ is a clique, cycle or complete bipartite graph.
\end{problem}

\subsection{Minimal elements of $\ca{F}^\ast(H)$}
In Section~\ref{sec:rel} we exploit a relation between $\prsat(n,H)$ and $\sat(n,\mathcal{F}^*(H))$, where $\mathcal{F}^*(H)$ is the family of graphs that contain a rainbow copy of $H$ in all proper colourings. Given this relationship, it is interesting to gain a better understanding of such families. 

For a family of (unlabelled) graphs $\ca{F}$, say a graph $F$ is a minimal element of $\ca{F}$ (by graph containment) if $F \in \ca{F}$ and no proper subgraph of $F$ is in $\ca{F}$.
\begin{observation}
Let $\ca{F}$ be a set of graphs, and let $\ca{M}$ the set of minimal elements of $\ca{F}$.
Then a graph is $\ca{F}$-saturated if and only if it is $\ca{M}$-saturated.
In particular, for all $n \in \bb{N}$, $\sat(n,\ca{F}) = \sat(n,\ca{M})$.
\end{observation}
Let $\ca{M}^\ast(H)$ be the set of minimal elements of $\ca{F}^\ast(H)$.
Note that each graph in $\ca{F}^\ast(H)$ contains at least one graph in $\ca{M}^\ast(H)$.
Conversely, if $A\subseteq \ca{M}^\ast(H)$ and each graph in $\ca{F}^\ast(H)$ contains at least one graph in $A$, then $A=\ca{M}^\ast(H)$.
\begin{observation}
Let $G,H$ be graphs.
The following are equivalent:
\begin{itemize}
    \item $G$ is properly rainbow $H$-saturated;
    \item $G$ is $\ca{F}^\ast(H)$-saturated;
    \item $G$ is $\ca{M}^\ast(H)$-saturated.
\end{itemize}
In particular, for all $n \in \bb{N}$, $\prsat(n,H) = \sat(n,\ca{F}^*(H))= \sat(n,\ca{M}^*(H))$.
\end{observation}
It is also worth noting that the analogous statements hold for the extremal number and rainbow extremal number, so $\rex(n,H) = \ex(n,\ca{F}^*(H))= \ex(n,\ca{M}^*(H))$.

One potentially fruitful approach to determining $\prsat(n,H)$ is first characterizing $\ca{M}^\ast(H)$, and then finding bounds for $\sat(n,\ca{M}^\ast)$.
Related to this, we pose the following question which is of independent interest.
\begin{question} \label{qst:minimal_finite}
For which graphs $H$ is $\ca{M}^\ast(H)$ finite?
\end{question}

The answer to Question \ref{qst:minimal_finite} seems to depend on the structure of the graph $H$.
Triangles and stars are rainbow in any proper colouring, so we trivially have $\ca{M}^\ast(K_3)=\{K_3\}$ and $\ca{M}^\ast(K_{1,k})=\{K_{1,k}\}$ for all $k \in \bb{N}$.
On the other hand, $\ca{M}^\ast(P_4)$ is infinite, as shown in the following result.
Let $T_5^*$ be the $5$-vertex graph obtained by subdividing one edge of the star $K_{1,3}$.

\begin{restatable}{theorem}{pathfourminimal}
$\ca{M}^*(P_4) = \{T_5^*\}\cup\{C_{2k+3}:k \in \bb{N}\}$.
\end{restatable}

\begin{proof}
In any proper colouring of $T_5^*$, letting $e$ be the edge not incident to the vertex of degree $3$ in $T_5^*$, at most one edge $e'\neq e$ receives the same colour as $e$, so removing this edge if necessary, we obtain a rainbow graph that contains $P_4$.
Thus, $T_5^* \in \ca{F}^*(P_4)$.
Ignoring isolated vertices, any proper subgraph of $T_5^\ast$ is either contained in $K_{1,3}$, which is $P_4$-free, or $P_4$, which has a rainbow $P_4$-free proper colouring, so $T_5^\ast \in \ca{M}^\ast(P_4)$.

Given $k\in \bb{N}$, any rainbow $P_4$-free proper colouring of $C_{2k+3}$ must be alternating, but odd cycles have no alternating proper colouring.
Any proper subgraph of $C_{2k+3}$ is contained in the path $P_{2k+3}$, which has an alternating (and thus rainbow $P_4$-free) proper colouring, so $C_{2k+3} \in \ca{M}^\ast(P_4)$.

Take any $G \in \ca{F}^\ast(P_4)$.
$G$ contains $P_4$, so $G$ is not a star.
Since a proper $3$-colouring of $K_4$ avoids a rainbow copy of $P_4$, we have $|V(G)|\ge 5$.
Thus, if $G$ has a vertex of degree at least $3$, then $G\supset T_5^\ast$.
Otherwise, $G$ is a disjoint union of cycles and paths, and $G$ has at least one component that is an odd cycle of order $\ge 5$, since $C_3 \not\supset P_4$ and path components can be coloured with alternating colourings to avoid a rainbow copy of $P_4$.
Thus, $G$ contains a graph in $ \{T_5^\ast\}\cup\{C_{2k+3}:k \in \bb{N}\}$.
This completes the proof.
\end{proof}

The following result illustrates that $\ca{M}^\ast(H)$ can also be finite in a non-trivial case. 
\begin{restatable}{theorem}{twomatchingminimal}
$\ca{M}^\ast(2K_2) = \{P_3\cup K_2\}$.
\end{restatable}
\begin{proof}
$P_3\cup K_2$ contains a rainbow copy of $2K_2$ in any proper colouring.
Ignoring isolated vertices, any proper subgraph of $P_3\cup K_2$ is contained in $2K_2$ and thus has a rainbow $2K_2$-free proper colouring.

Let $G \in \ca{F}^\ast(2K_2)$; we can assume without loss of generality that $G$ has no isolated vertices.
A proper $3$-colouring of $K_4$ avoids a rainbow copy of $2K_2$, so $|V(G)|\ge 5$.
$G$ is not a star, else $G$ is $2K_2$-free.
$G$ is not a matching, else it is properly $1$-colourable.
Thus, if $G$ is disconnected, then it has a vertex $v$ of degree $2$ and a component disjoint from the one containing $v$, so it contains $P_3 \cup K_2$.
If $G$ is connected, then since it is not a star, it contains $P_5$ or $T_5^\ast$, both of which contain $2K_2$.
Thus, $G\supset P_3 \cup K_2$.
This completes the proof.
\end{proof}

\subsection{Determining when $\prsat(n,H) = O(1)$}

In Section~\ref{sec:isolated}, we determined several families of graphs $H$ such that $\prsat(n,H \cup K_2) = O(1)$.

A natural question to ask is the following.

\begin{question} \label{qst:constant_satast}
Is it the case that if $H$ has an isolated edge, then $\prsat(n,H) = O(1)$?
\end{question}

Perhaps a reasonable starting point is to consider trees. Proper rainbow saturation is particularly well-understood for trees: in~\cite{trees}, we provide an asymptotic characterization of the proper rainbow saturation number for a large class of trees.

\begin{question}
    Let $T$ be a tree. Is $\prsat(n,T \cup K_2) = O(1)$?
\end{question}

Given Theorem~\ref{thm:mK2}, it is natural to wonder whether adding an isolated edge to a graph with constant proper rainbow saturation number yields another such graph.

\begin{question}
    Let $H$ be a graph such that $\prsat(n,H) = O(1)$. Is $\prsat(n, H \cup K_2) = O(1)$?
\end{question}

It would also be interesting to find other cyclic graphs for which this is true, as the triangle is rainbow in any proper colouring so does not provide a particularly interesting example. 

\textbf{Note added before submission:} Whilst finalising this manuscript, it came to our attention that Theorem~\ref{thm:K4construction} and Theorem~\ref{thm:Ck_upper_satast}(i) had simultaneously and independently been proved by Baker, Gomez-Leos, Halfpap, Heath, Martin, Miller, Parker, Pungello, Schwieder, and Veldt~\cite{emily}. They additionally proved that the bound in Theorem~\ref{thm:K4construction} is asymptotically tight, and they provide a general lower bound of $\prsat(n,K_k) \ge (k-1)n+O(1)$ for $k \ge 5$ and an upper bound of $\prsat(n,K_k) \le \frac{k^4}{8}+O(1)$ for $k\ge 3$, among other interesting results.

\bibliographystyle{amsplain}
\bibliography{RainbowSat}

\end{document}